\documentclass[a4paper,12pt]{article}

\usepackage[intlimits]{amsmath}
\usepackage{amssymb}
\usepackage{amsthm}
\usepackage{mathrsfs}
\usepackage{cite}
\usepackage{bbm}
\usepackage{mathrsfs}
\usepackage{color}
\usepackage{verbatim}

\newcommand{\la}{\langle}
\newcommand{\ra}{\rangle}
\newcommand{\pr}{\partial}
\newcommand{\dom}{\Omega}
\newcommand{\sig}{\Sigma}
\newcommand{\gam}{\Gamma}

\newcommand{\R}{\mathbb{R}}

\newcommand{\dd}{\,\text{d}}
\newcommand{\supp}{\mbox{supp\;}}

\newcommand{\cc}{\mathbf{C}}

\newcommand{\hb}{H^{1/2}_{co}}
\newcommand{\hbi}{H^{-1/2}_{co}}

\newtheorem{thm}{Theorem}[section]
\newtheorem{prop}{Proposition}[section]
\newtheorem{lem}{Lemma}[section]
\newtheorem{cor}{Corollary}[section]

\newtheorem{remark}{Remark}[section]

\title{Uniqueness for the inverse boundary value problem of piecewise homogeneous anisotropic elasticity in the time domain}

\author{C\u{a}t\u{a}lin I. C\^{a}rstea\thanks{School of Mathematics, Sichuan University, Chengdu, Sichuan, 610064, P.R.China; email: catalin.carstea@gmail.com}\and  Gen Nakamura\thanks{Hokkaido University, Sapporo  060-0808, Japan; email: gnaka@math.sci.hokudai.ac.jp}\and Lauri Oksanen\thanks{University College London, London, UK; email: l.oksanen@ucl.ac.uk} 
}

\date{}

\begin{document}
\maketitle

\begin{abstract}
We consider the inverse boundary value problem of recovering piecewise homogeneous elastic tensor and piecewise homogeneous mass density from a localized lateral Dirichlet-to-Neumann or Neumann-to-Dirichlet map for the elasticity equation in the space-time domain. We derive uniqueness for identifying these tensor and density on all domains of homogeneity that may be reached from the part of the boundary where the measurements are taken by a chain of subdomains whose successive interfaces contain a curved portion.
\end{abstract}

{\bf Keywords.} Inverse boundary value problem; uniqueness; anisotropic elasticity.\\

{\bf MSC(2000): } 35R30, 35L10

\section{Introduction}\label{introduction}

Inverse boundary value problems are concerned with the determination of the physical properties (represented as various coefficients of a model equation) of an object from measurements taken on the boundary. The mathematical investigation of such problems already has a history spanning nearly four decades, going back to \cite{Ca}. In this context, uniqueness refers to the property certain equations might have that if two sets of coefficients produce identical sets of boundary data, then the coefficients must also be identical.

For second order elliptic equations, the usual boundary data used is the set of Dirichlet and Neumann boundary values of the  solutions of the equation. Physically, these may be interpreted as measurements of electric potential and current density in the case of the conductivity equation for the electric impedance tomography, or as measurements of displacement and traction in the case of the static elasticity equation for the nondestructive testing. One common way to encode the measurement data is as the Dirichlet-to-Neumann map, which associates to the one Dirichlet datum the corresponding Neumann one. It is also possible to consider ``local data'', i.e. Dirichlet data supported on a part of the boundary paired with the restriction of the corresponding Neumann data to that same part.

In the case of isotropic materials, many uniqueness results are known for these type of measurement data. See, for example, \cite{SU} for the case of isotropic conductivity, or \cite{NU-3} for isotropic elasticity.  For anisotropic materials uniqueness usually doesn't hold. In the case of the anisotroipc conductivity equation, for example, diffeomorphisms that leave the boundary fixed preserve the Dirichlet-to-Neumann map but change the interior conductivity. It is conjectured that this is the only obstruction. If one knows apriori that the coefficients belong to a more restricted class, which is not preserved by diffeomorphisms, then it can be expected that uniqueness may  hold. One such instance (for example for the anisotropic conductivity or anisotropic elasticity cases) would be the the class of piecewise constant coefficients. Uniqueness for the anisotropic conductivity case with this restriction was shown in \cite{AHG}. For the (static) anisotropic elasticity equation, it was shown in \cite{CHN}.

For second order hyperbolic equations, lateral Dirichlet and Neumann data may be used. That is, time dependent Dirichlet and Neumann data measured on the boundary of the object, for solutions with zero initial Cauchy data. Uniqueness has been proved for various second order hyperbolic equations (for example, see \cite{B}, \cite{B-3}, \cite{KKL}). These results are based on the so-called Boundary Control method. Computational aspects of this method have been studied recently in
\cite{BIKS}, \cite{hoop}, \cite{PBK}, see also \cite{BG} for the first computational
implementation of the method. Uniqueness up to diffeomorphism in anisotropic materials has been shown in \cite{B-3}, \cite{BK}. There is only one uniqueness result known for a time domain anisotropic elasticity equation with smooth coefficients. That is for the hexagonal elasticity equation with ellipsoidal slowness surfaces (see \cite{Mazzucato}). In \cite{HNZ} a uniqueness result is proved for the piecewise constant elastic coefficients in the generally anisotropic case (or piecewise analytic, but with greater symmetry assumptions). 

The argument showing the uniqueness in \cite{HNZ}, as is typical for the above types of measurement data setups, consists of two steps. The first step is a ``boundary determination'' result. This is accomplished via a finite in time Laplace transform which relates the lateral Dirichlet-to-Neumann map of the hyperbolic problem to the symbol of the Dirichlet-to-Neumann map of an elliptic equation of elastic type. Using the techniques from \cite{NTU}, \cite{NU}, \cite{NU-2}, the coefficients at the boundary may be obtained from the given (local) data. The second step is an ``interior determination'' result. That is, for the above mentioned elliptic equation, it is shown that the local Dirichlet-to-Neumann on the boundary determines the local Dirichlet-to-Neumann map for the domain with the immediately adjacent domain of homogeneity removed, localized on a previously internal part of the boundary. This technique, inspired by \cite{I}, was used for interior determination in \cite{CHN}. It relies essentially on the Runge approximation property. Once these two steps are completed, the uniqueness result follows by induction. 

In their ``boundary determination'' result, there is no restriction on the observation time. However, it is important to remark that the full elliptic Dirichlet-to-Neumann map is needed for the ``interior determination'' step, not just its symbol. This then introduces the restriction that the observation time has to be sufficiently long. For this type of method, this is to be expected since, due to the finite speed of propagation property of the equation, we expect to need to wait for a certain amount of time before the waves generated by the boundary data can travel to every part of the domain of interest and then return to the part of the boundary where observations are made. Only after this large amount of time has passed enough measurement data can be obtained which enable to do the determination of the coefficients in the whole domain. However, this is not the kind of result that should be expected for a hyperbolic equation. In the hyperbolic case, the boundary data gathered for any length of time should provide information on the material at least in the part of the domain from which the elastic waves have enough time to return to the surface on which the data is collected. One of the aims of this paper is to provide this type of result.

We will consider  the inverse boundary value problem for the  piecewise homogeneous, generally anisotropic, elasticity equation in the space-time domain and prove the above type of uniqueness result. We will use the boundary determination result of \cite{HNZ}. In Section \ref{boundary} we sketch the connection between the time domain equation and the elliptic one via the finite time Laplace transform and then quote the relevant boundary uniqueness result from \cite{CHN},\cite{HNZ},\cite{NTU}.

We will not use the elliptic equation for the interior determination part of our result. Instead, we use hyperbolic techniques inspired by \cite{KOP}. This allows us to obtain a uniqueness result that is more typical for hyperbolic equations, i.e. \emph{we obtain uniqueness closer to the part of the boundary where the measurements are taken for shorter observation times and  further away from that part of the boundary for longer times}.

We would like to remark on the importance of showing uniqueness in the inverse boundary value problem for the local Neumann-to-Dirichlet map. The vibroseis exploration technique in the reflection seismology is used to investigate the underground structure of the Earth. The measurements taken using this technique correspond almost precisely to this kind of map, and not to a local Dirichlet-to-Neumann map (see \cite{Bateman}). Furthermore, the geological structure of the Earth is isotropic or transversally isotropic layered in its shallow part, and in its further depth part, it is close to piecewise homogeneous and can contain regions with more complicated anisotropic elasticity than transversally isotropic elasticity.

The rest of this paper is organized as follows. In what remains of Section \ref{introduction}, we formulate our inverse  problem and give our main results on the unique identification of piecewise homogeneous density and elasticity tensor. 
The first main results is the unique identification for the case when the geometry of the homogeneous pieces is known. The second main result is in  the case when the geometry of this pieces is unknown. In Section \ref{boundary} we briefly review a boundary determination result that has been obtained previously in \cite{CHN}, \cite{HNZ}. In Section \ref{inner} we discuss interior determination results for both Dirichlet-to-Neumann and Neumann-to-Dirichlet maps. Then in Section \ref{proof}, by combining these the results of the previous two sections, we complete the proofs of our main results, in the second case making use of the theory of subanalytic sets. Section A which is an appendix provides a brief summary of results on the theory of subanalytic sets which are used in the proof of the second main result.

\subsection{Preliminaries}

Let $\dom\subset\R^3$ be an  open bounded connected domain, with Lipschitz boundary. For a $T>0$ we will denote $\dom_T=(0,T)\times\dom$ an $\Gamma=\pr\dom$.

An \emph{elastic tensor} $\cc=\cc(x)=(C_{ijkl}(x))_{i,j,k,l=1,2,3}$, $x=(x_1,x_2,x_3)\in\overline\dom$, is defined by real valued functions  $C_{ijkl}(x)$ which satisfy the symmetries
\begin{equation}
C_{ijkl}(x)=C_{ijlk}(x),\quad C_{ijkl}(x)=C_{klij}(x),\quad x\in\overline\Omega,\,\, i,j,k,l\in\{1,2,3\},
\end{equation}
and the strong convexity condition, i.e. there exists $\lambda>0$ such that for any symmetric matrix $\epsilon=(\epsilon_{ij})$,
\begin{equation}
\epsilon:(\cc:\epsilon)=\sum_{i,j,k,l=1}^3 C_{ijkl}(x)\epsilon_{ij}\epsilon_{kl}\ge \lambda (\epsilon:\epsilon),\,\,x\in\overline\Omega,
\end{equation}
where $\epsilon:\eta$ is the inner product of matrices $\epsilon$ and $\eta=(\eta_{ij})$ defined by
$\epsilon:\eta=\sum_{i,j=1}^3\epsilon_{ij}\eta_{ij}$, and $\cc:\eta$ is a matrix whose $(i,j)$ component $(\cc:\eta)_{ij}$ is defined as
$(\cc:\eta)_{ij}=\sum_{k,l=1}^3 C_{ijkl}(x)\eta_{kl}$. The density of mass function is given as a function $\rho(x)>\lambda$, $x\in\overline\dom$, which will be simply called \emph{density}.

In this paper we will make the further assumption that there are a finite number of open, connected,  Lipschtiz subdomains $D_\alpha$, $\alpha\in A$, that is subdomains with Lipschitz boundaries, such that $\bar\dom=\cup_{\alpha\in A}\bar D_\alpha$, $D_\alpha\cap D_\beta=\emptyset$ if $\alpha\neq\beta$, and $\cc$, $\rho$  are homogeneous in each $D_\alpha$.

For a function $u:\dom\to\R^3$ denoting the diplacement we write
\begin{equation}
(L_\cc u)_i = \sum_{j,k,l=1}^3\pr_j \left( C_{ijkl} \pr_k u_l\right), 
\end{equation}
where $\partial_j=\partial_{x_j}$
and for $u:\Omega_T\to \R^3$ we write
\begin{equation}
(P_{\rho,\cc} u)_i = \rho\pr_t^2 u_i - (L_\cc u)_i.
\end{equation}
Let $\nu$ be the unit outer normal to $\pr\dom$. We will denote the traction at $\partial\Omega$ by
\begin{equation}
(\pr_\cc u)_i:=[ (\cc:Du)\nu]_i=\sum_{j,k,l=1}^3\nu_jC_{ijkl}\pr_ku_l.
\end{equation}

\subsection{The forward problem and local boundary data}

We want to consider in this paper the equations
\begin{equation}
\left\{\begin{array}{l}
P_{\rho,\cc}u=0\,\,\text{in}\,\,\dom_T,\\[5pt]
(u,\pr_t u)|_{t=0}=0,
\end{array}\right.
\end{equation}
with Dirichlet or Neumann boundary conditions. At one point in the argument we will however need a result with slightly more general boundary conditions, namely with both Dirichlet and Neumann data given on complementary parts of the boundary. This kind of problems have been considered in the case of piezoelectric equations in \cite{AN}. Here we will give the restriction of their result to the case of hyperbolic elasticity equations.

We need to introduce a few function spaces. Suppose $\sig$ is an open subset of $\gam$. Let 
\begin{equation}
    H^{\pm1/2}_{co}(\sig)=\left\{ f\in H^{\pm1/2}(\gam): \supp g\subset\overline{\sig}\right\}.
\end{equation} 
with the restriction of the norm of $H^{\pm1/2}(\gam)$, and
\begin{equation}
    H^{\pm1/2}(\sig)=\left\{f|_\sig:f\in H^{\pm1/2}(\gam)\right\},
\end{equation}
with the norm
\begin{equation}
    ||f||_{H^{\pm1/2}(\sig)}=\inf_{\tilde f,\,\tilde f|_{\sig}=f}||\tilde f||_{H^{\pm1/2}(\gam)}.
\end{equation}

Suppose $\Gamma_u,\Gamma_s\subset\Gamma$  are disjoint connected open sets such that $\pr\gam_u=\pr\gam_s$ is a Lipschitz curve. We write
\begin{equation}
H^1_{\Gamma_u}(\dom)=\left\{u\in H^1(\dom): u|_{\Gamma_u}=0\right\}\,\,\text{and}\,\,(H^1_{\Gamma_u}(\dom))'\,\,\text{for its dual space}.
\end{equation}

Let $f\in H^3(0,T;H^{1/2}(\Gamma_u))$, $g\in H^1(0,T;H^{-1/2}(\Gamma_s))$ be boundary data, $F\in H^1(0,T;(H^1_{\Gamma_u}(\dom))')$ be a source term, and $u_0\in H^1(\dom)$, $u_1\in L^2(\dom)$ be initial data satisfying the compatibility condition
\begin{equation}
f(0,\cdot)=u_0|_{\Gamma_u}.
\end{equation}
Then we have the following theorem.
\begin{thm}\label{thm-forward}
There exists a unique $u\in L^\infty(0,T;H^1(\dom))$ with $\pr_t u\in L^\infty(0,T;L^2(\dom))$, $\pr_t^2u\in L^\infty(0,T;(H^1_{\Gamma_u}(\dom))')$ such that
\begin{equation}
\left\{\begin{array}{l}
P_{\rho,\cc}u=F\,\,\text{\rm in}\,\,\dom_T,\\[5pt]
u=f\,\,\text{\rm on}\,\,(0,T)\times\Gamma_u,\\[5pt]
\pr_\cc u=g\,\,\text{\rm on}\,\,(0,T)\times\Gamma_s\\[5pt]
u|_{t=0}=u_0,\;\pr_t u|_{t=0}=u_1,
\end{array}\right.
\end{equation}
and
\begin{multline}
||u||_{L^\infty(0,T;H^1(\dom))}+||\pr_t u||_{L^\infty(0,T;L^2(\dom))}+||\pr_t^2 u||_{L^\infty(0,T;(H^1_{\Gamma_u}(\dom))')}\\[5pt]
\leq
C(||u_0||_{H^1(\dom;\R^3)}+||u_1||_{L^2(\dom)}
+||F||_{H^1(0,T;(H^1_{\Gamma_u}(\dom))')}\\[5pt]
+||f||_{H^3(0,T;H^{1/2}(\Gamma_u))}
+||g||_{H^1(0,T;H^{-1/2}(\Gamma_s))}),
\end{multline}
where the constant $C$ depends on $\dom$, $T$, $\lambda$, $||\cc||_{L^\infty(\dom)}$, $||\rho||_{L^\infty(\dom)}$.
\end{thm}

\begin{remark}One reason why the source term $F$ has higher regularity in time than is usually assumed is the fact that it is more singular in space. We will in fact at one point in the proof need to solve the equation with a source term that would have exactly this regularity. The Dirichlet data is required to have three derivatives in time because, in the course of the proof of the result, it is necessary to convert the given equation to one with zero Dirichlet data, but with an additional source term, which has to have $H^1$ in time regularity.
\end{remark}

We can define the local Neumann-to-Dirichlet and Dirichlet-to-Neumann maps, which contain the information that may be collected by applying various tractions to $\sig$ and measuring displacements on $\sig$, or by producing displacements and measuring tractions on $\sig$ as follows. The local Neumann-to-Dirichlet map (abbreviated as ``ND map'')
\begin{equation}
\Phi^{T,\sig}_{\rho,\cc}: H^1(0,T;\hbi(\sig))\to H^1(0,T;H^{1/2}(\sig))
\end{equation} 
is defined by
\begin{equation}
\Phi^{T,\sig}_{\rho,\cc}g=u|_{{(0,T)\times\sig}},\,\,\text{\rm with}\,\,{(0,T)\times\sig}=(0,T)\times\Sigma,
\end{equation}
where $u$ solves
\begin{equation}
\left\{\begin{array}{l}
P_{\rho,\cc}u=0\,\,\text{\rm in}\,\,\dom_T,\\[5pt]
\pr_\cc u|_{{(0,T)\times\sig}}=g,\\[5pt]
(u,\pr_t u)|_{t=0}=0.
\end{array}\right.
\end{equation}
The local Dirichlet-to-Neumann map (abbreviated as ``DN map'') 
\begin{equation}
\Lambda^{T,\sig}_{\rho,\cc}: H^3(0,T;\hb(\sig))\to H^1(0,T;H^{-1/2}(\sig))
\end{equation} 
is defined by
\begin{equation}
\Lambda^{T,\sig}_{\rho,\cc}f=\pr_\cc u|_{{(0,T)\times\sig}},
\end{equation}
where $u$ solves
\begin{equation}
\left\{\begin{array}{l}
P_{\rho,\cc}u=0\,\,\text{\rm in}\,\,\dom_T,\\[5pt]
u|_{{(0,T)\times\sig}}=f,\\[5pt]
(u,\pr_t u)|_{t=0}=0.
\end{array}\right.
\end{equation}

\subsection{Main results}

If  $\sig\subset\partial D_\alpha$ is open, we will say that $\sig$ is {\it curved} if it is $C^1$ and $\{\nu(x):x\in\sig\}\subset S^2$ contains the image of a non-constant continuous curve.

Suppose $\cc^{(I)}$, $\rho^{(I)}$ $I=1,2$, are two elastic tensors and two densities on $\dom$, all of which are homogeneous in common Lipschitz subdomains $D_\alpha$. Since each boundary $\partial D_\alpha$ of $D_\alpha$ can have the discontinuity of the density and elastic tensor, we also call each $\partial D_\alpha$ interface. We will consider a chain $D_{\alpha_i}$, $i=1,\ldots,N$ of these subdomains (which we will abbreviate as $D_i$) and nonempty open surfaces $\sig_i\subset\pr D_i$ such that $\sig=\sig_1\subset\gam\cap\pr D_1$, and $\sig_{i+1}\subset\bar D_i\cap\bar D_{i+1}$, $i=1,\ldots,N-1$.

\begin{thm}\label{main-thm}
Let $\cc^{(1)}$, $\cc^{(2)}$, $\rho^{(1)}$, $\rho^{(2)}$ be as above, assume that each $\sig_i$, $i=1,\ldots,N-1$ is curved in the sense given above, and $\pr(\pr\dom\cap\pr D_1)$, $\pr(\pr D_i\cap\pr D_{i+1})$, $i=1,\ldots,N-1$ are Lipschitz curves, then there exist times $0<T_1<\cdots<T_N<\infty$, such that if
\begin{equation}
\Lambda^{T_k,\sig}_{\rho^{(1)},\cc^{(1)}}=\Lambda^{T_k,\sig}_{\rho^{(2)},\cc^{(2)}},
\end{equation}
or if
\begin{equation}
\Phi^{T_k,\sig}_{\rho^{(1)},\cc^{(1)}}=\Phi^{T_k,\sig}_{\rho^{(2)},\cc^{(2)}},
\end{equation}
then
\begin{equation}
\rho^{(1)}|_{D_i}=\rho^{(2)}|_{D_i},\;\cc^{(1)}|_{D_i}=\cc^{(2)}|_{D_i},\quad i=1,\ldots,k.
\end{equation}
\end{thm}
\begin{remark}
The condition that $\pr(\pr\dom\cap\pr D_1)$, $\pr(\pr D_i\cap\pr D_{i+1})$, $i=1,\ldots,N-1$ are Lipschitz curves is in fact not necessary. Given a chain of domains that satisfies the other assumptions, we can pick a smooth curve that connects a point in $\sig$ to a point inside $D_N$ and crosses each interface $\pr D_i$ transversely. The intersection of the original chain of domains with an appropriately chosen tubular neighborhood of this curve would satisfy this extra condition.
\end{remark}
\begin{remark}
Note that in the case when $\rho|_{D_i}$, $\cc|_{D_i}$ are known for some $i$, then we do not need to assume that $\gam_i$ is curved. This would permit our result to apply, for example, to measurements of elastic waves taken on the surface of the Earth, which is locally very close to flat, as long as the properties of the top layer of the ground are known by other means. If the underground regions whose elastic properties are unknown can be reached by passing through a number of interfaces that do have curved portions, then the result still holds.
\end{remark}

\bigskip
Suppose $\cc^{(I)}$, $\rho^{(I)}$ $I=1,2$, are two elastic tensors and two densities on $\dom$, all of which are homogeneous on Lipschitz subdomains $D_\alpha^{(I)}$. Suppose further that there is  a  ``region of interest'' $R\subset\dom$, $\sig\subset\pr R$, such  that any $D_\alpha^{(I)}\cap R$ is sub-analytic.\footnote{For a definition and summary of properties of sub-analytic sets, see Appendix \ref{sas}} 

\begin{thm}\label{thm-interest}
Let $\cc^{(1)}$, $\cc^{(2)}$, $\rho^{(1)}$, $\rho^{(2)}$, be as above, and assume that $\sig$ is curved and that any boundary $(\pr D_\alpha^{(I)}\cap R)\setminus\partial R$  is curved on all its smooth components. Then there exists a time $0<T<\infty$ so that if 
\begin{equation}
\Lambda^{T,\sig}_{\rho^{(1)},\cc^{(1)}}=\Lambda^{T,\sig}_{\rho^{(2)},\cc^{(2)}},
\end{equation}
or if 
\begin{equation}
\Phi^{T,\sig}_{\rho^{(1)},\cc^{(1)}}=\Phi^{T,\sig}_{\rho^{(2)},\cc^{(2)}},
\end{equation}
then
\begin{equation}
\rho^{(1)}|_{R}=\rho^{(2)}|_{R},\;\cc^{(1)}|_{R}=\cc^{(2)}|_{R}.
\end{equation}
\end{thm}

\section{Boundary determination}\label{boundary}

In this section we follow \cite{HNZ} to show that the DN map $\Lambda^{T,\sig}_{\rho,\cc}$, or the ND map $\Phi^{T,\sig}_{\rho,\cc}$, determines the values of $\rho|_\sig$ and $C|_\sig$. More precisely we will show the following.
\begin{prop}[see {\cite[Theorem 5.4]{HNZ}}]\label{boundary-prop}
If $\sig$ is curved, $0<T<\infty$, and $\Lambda^{T,\sig}_{\rho^{(1)},\cc^{(1)}}=\Lambda^{T,\sig}_{\rho^{(2)},\cc^{(2)}}$ or $\Phi^{T,\sig}_{\rho^{(1)},\cc^{(1)}}=\Phi^{T,\sig}_{\rho^{(2)},\cc^{(2)}}$, then $\rho^{(1)}|_{D_1}=\rho^{(2)}|_{D_1}$ and $\cc^{(1)}|_{D_1}=\cc^{(2)}|_{D_1}$.
\end{prop}

\begin{proof}
We will sketch the argument in the DN map and the ND map cases separately.

\paragraph{The DN map case:}
This is proved in \cite{HNZ}. Here we will repeat enough of the argument to give the reader an idea of how it works, but we will not reproduce it in full. The main idea is to use a finite-time Laplace transform in order to convert the hyperbolic elasticity boundary value problem to an elliptic boundary value problem. The boundary determination will then follow from the results on the elliptic case proven in \cite{CHN}, \cite{HNZ}.

Suppose $\cc$ and $\rho$ are of the kind we are considering. Let $\psi\in\hb(\sig)$ and let  $u$  be a solution of
\begin{equation}
\left\{\begin{array}{l}
P_{\rho,\cc} u=0\,\,\text{\rm in}\,\,\dom_T,\\[5pt]
 u|_{(0,T)\times\gam}=t^2\psi,\\[5pt]
  (u, \pr_t u)|_{t=0}=0.
\end{array}\right.
\end{equation}
Also for $\phi\in\hb(\sig)$, we will consider the elliptic boundary value problem 
\begin{equation}\label{elliptic}
\left\{\begin{array}{l}
\rho v-h^2L_\cc v=0\,\,\text{\rm in}\,\,\dom,\\[5pt]
v|_{\gam}=\phi
\end{array}\right.
\end{equation}
depending on a parameter $h>0$.
Let
\begin{equation}
\tilde\Lambda^{h,\sig}_{\rho,\cc}( \phi)=h\pr_\cc v|_{\sig}
\end{equation}
be the associated DN map.

We are interested in comparing $v(\cdot,h)$ to the finite-time Laplace transform 
\begin{equation}
(\mathcal{L}_T u)(\cdot,h) := \int_0^T u(t,\cdot)e^{-\frac{t}{h}}\dd t,
\end{equation}
 where the Dirichlet data for $v$ is chosen so that 
\begin{equation}\label{def-phi}
\phi=\psi\int_0^T t^2e^{-\frac{t}{h}}\dd t.
\end{equation}

Let $u_0\in H^1(\dom)$ be the solution of
\begin{equation}
\left\{\begin{array}{l}
L_\cc u_0=0\,\,\text{\rm in}\,\,\dom,\\[5pt]
u_0|_{\gam}=\psi,
\end{array}\right.
\end{equation}
and define $u_1$ so that
\begin{equation}
u(t,x)=t^2u_0(x)+u_1(t,x),\quad (t,x)\in \dom_T.
\end{equation}
Then $u_1$ should satisfy
\begin{equation}
\left\{\begin{array}{l}
P_{\rho,\cc}u_1=-2\rho u_0\,\,\text{\rm in}\,\,\dom_T\\[5pt]
 u_1|_{(0,T)\times\gam}=0,\\[5pt]
  (u_1, \pr_t u_1)|_{t=0}=0.
\end{array}\right.
\end{equation}
It is known that by the Korn inequality, there exists a unique solution such that
\begin{equation}
u_1\in C([0,T];H^1(\dom))\cap C^1([0,T]; L^2(\dom))
\end{equation}
and
\begin{multline}
||u_1(t)||_{H^1(\dom)}+||\pr_t u_1(t)||_{L^2(\dom)}\\[5pt]\leq C||2\rho u_0||_{L^2(\dom_T)}
\leq C||\psi||_{H^{1/2}(\dom)},\,\,t\in[0,T]
\end{multline}
(see \cite{Wloka}).

Let
\begin{equation}
r(\cdot,h)=v(\cdot,h)-\int_0^Tu(t,\cdot)e^{-\frac{t}{h}}\dd t.
\end{equation}
An elementary computation shows that $r$ satisfies
\begin{equation}
\left\{
\begin{array}{l}
h^{-2}\rho r -L_\cc r = e^{-\frac{T}{h}}\left[\pr_t u_1(T)+h^{-1} u_1(T)
+\rho u_0(T^2h^{-1}+2T)\right]\,\,\text{\rm in}\,\,\Omega,\\[5pt]
r|_{\gam}=0.
\end{array}
\right.
\end{equation}
Let $\tilde T=\max(1,T)$.
By the standard elliptic estimates it follows that
\begin{equation}
||r||_{H^1(\dom)}\leq Ch^{-1}\tilde T^2e^{-\frac{T}{h}}||\psi||_{\hb(\sig)},
\end{equation}
for $0<h<1$, and with a constant $C>0$ independent on $T$ or $h$.

From \eqref{def-phi} it follows that $t^2\phi=\chi\psi$, where
\begin{equation}
\chi(t,T,h)=t^2\left(\int_0^T s^2 e^{-\frac{s}{h}}\dd s\right)^{-1}.
\end{equation}
It is easy to see that there exists $h_0>0$ so that if $0<h<h_0$, then
\begin{equation}
\chi(t,T,h)<CT^2h^{-3},\quad 0 \leq t\leq T,
\end{equation}
where the constant $C>0$ again is independent of $T$ or $h$. 
We can conclude that
\begin{equation}
||\tilde\Lambda^{h,\sig}_{\rho,\cc} \phi-h\mathcal{L}_T\Lambda^{T,\sig}_{\rho,\cc}(\chi\phi)||_{H^{-1/2}(\sig)}\leq
C\tilde T^4 h^{-3} e^{-\frac{T}{h}}||\phi||_{\hb(\sig)},
\end{equation}
or
\begin{equation}
||\tilde\Lambda^{h,\sig}_{\rho,\cc} -h\mathcal{L}_T\Lambda^{T,\sig}_{\rho,\cc}\chi||_{\hb(\sig)\to H^{-1/2}(\sig)}\leq
C\tilde T^4 h^{-3} e^{-\frac{T}{h}},
\end{equation}
where the constant $C>0$ is independent of $T$ or $h$.

Considering $\tilde\Lambda^{h,\sig}_{\rho,C}$ as a semiclassical pseudodifferential operator with the small parameter $h$, it follows that its full symbol can be obtained from $\Lambda^{T,\sig}_{\rho,C}$ (but not necessarily the operator itself).

It is shown in \cite{HNZ} (see their Theorem 4.2) that the principal symbol of $\tilde\Lambda^{h,\sig}_{\rho,C}$ determines $\Gamma(x,h),\,x\perp\nu$, where $\Gamma(x,h)$ is the fundamental solution of $\rho-h^2L_C$ associated to the pair $\rho|_{D_1}$ and $C|_{D_1}$ whose Fourier transform with respect to $x$ with $x\perp\nu$ is bounded as $h\rightarrow 0$. See also \cite{NTU}, \cite{T}, for similar results. Once having this, it can be shown that by using $\sig$ is curved, we can recover $\rho|_{D_1}$ and $C|_{D_1}$. This is shown in \cite[Apendix B]{HNZ}.

\paragraph{The ND map case:}
The method of proof is almost entirely analogous to the DN map case. We will give only a brief sketch of its argument. 

Let $\psi\in\hbi(\sig)$ and let  $u$  be a solution of
\begin{equation}
\left\{\begin{array}{l}
P_{\rho,\cc} u=0\,\,\text{\rm in}\,\,\dom_T,\\[5pt]
 \pr_\cc u|_{(0,T)\times\gam}=t^2\psi,\\[5pt]
  (u, \pr_t u)|_{t=0}=0.
\end{array}\right.
\end{equation}
For $\phi\in\hbi(\sig)$ consider the elliptic boundary value problem 
\begin{equation}\label{elliptic-N}
\left\{\begin{array}{l}
\rho v-h^2L_\cc v=0\,\,\text{\rm in}\,\,\dom,\\[5pt]
\pr_\cc v|_{\gam}=\phi
\end{array}\right.
\end{equation}
depending on a parameter $h>0$.
Let
\begin{equation}
\tilde\Phi^{h,\sig}_{\rho,\cc} \phi=h^{-1}v|_{\sig}
\end{equation}
be the associated ND map.

We will choose the Neumann data $\phi$ as
\begin{equation}
\phi=\psi\int_0^T t^2e^{-\frac{t}{h}}\dd t
\end{equation}
and let $u_0\in H^1(\dom)$ be the solution of
\begin{equation}
\left\{\begin{array}{l}
L_\cc u_0=0\,\,\text{\rm in }\,\,\dom,\\[5pt]
\pr_\cc u_0|_{\gam}=\psi.
\end{array}\right.
\end{equation}
Likewise before for the DN map case, we define $u_1$ so that
\begin{equation}
u(t,x)=t^2u_0(x)+u_1(t,x),\quad (t,x)\in \dom_T.
\end{equation}
Then $u_1$ should satisfy
\begin{equation}
\left\{\begin{array}{l}
P_{\rho,\cc}u_1=-2\rho u_0\,\,\text{\rm in }\,\,\dom_T\\[5pt]
 \pr_\cc u_1|_{(0,T)\times\gam}=0,\\[5pt]
  (u_1, \pr_t u_1)|_{t=0}=0.
\end{array}\right.
\end{equation}
By the Korn inequality, this equation has a unique solution such that
\begin{equation}
u_1\in C([0,T];H^1(\dom))\cap C^1([0,T]; L^2(\dom))
\end{equation}
and
\begin{multline}
||u_1(t)||_{H^1(\dom)}+||\pr_t u_1(t)||_{L^2(\dom)}\\[5pt]
\leq C||2\rho u_0||_{L^2(\dom_T)}\leq C||\psi||_{H^{-1/2}(\dom)},\,\,t\in[0,T]
\end{multline}
(see \cite{Wloka}).

Now let
\begin{equation}
r(\cdot,h)=v(\cdot,h)-\int_0^Tu(t,\cdot)e^{-\frac{t}{h}}\dd t.
\end{equation}
It satisfies
\begin{equation}
\left\{\begin{array}{l}
h^{-2}\rho r -L_\cc z = e^{-\frac{T}{h}}\left[\pr_tu_1(T)+h^{-1} u_1(T)
+\rho u_0(T^2h^{-1}+2T)\right]\,\,\text{\rm in}\,\,\Omega,\\[5pt]
\pr_\cc r|_{\gam}=0.
\end{array}\right.
\end{equation}
Let $\tilde T=\max(1,T)$.
By the standard elliptic estimates it follows that
\begin{equation}
||r||_{H^1(\dom)}\leq Ch^{-1}\tilde T^2e^{-\frac{T}{h}}||\psi||_{\hbi(\sig)},
\end{equation}
for $0<h<1$ and with a constant $C$ independent on $T$ or $h$.

Similarly to the DN case, we can conclude that
\begin{equation}
||\tilde\Phi^{h,\sig}_{\rho,\cc} \phi-h^{-1}\mathcal{L}_T\Phi^{T,\sig}_{\rho,\cc}(\chi\phi)||_{H^{1/2}(\sig)}\leq
C\tilde T^4 h^{-5} e^{-\frac{T}{h}}||\phi||_{\hbi(\sig)},
\end{equation}
or
\begin{equation}
||\tilde\Phi^{h,\sig}_{\rho,\cc} -h^{-1}\mathcal{L}_T\Phi^{T,\sig}_{\rho,\cc}\chi||_{\hbi(\sig)\to H^{1/2}(\sig)}\leq
C\tilde T^4 h^{-5} e^{-\frac{T}{h}},
\end{equation}
where the constant $C$ is independent of $T$ or $h$.

As above, we can obtain the symbol of $\tilde\Phi^{h,\sig}_{\rho,\cc}$ from $\Phi^{T,\sig}_{\rho,\cc}$. The principal symbol of $\tilde\Lambda^{h,\sig}_{\rho,\cc}$ is the inverse of the symbol of $\tilde\Phi^{h,\sig}_{\rho,\cc}$. We can therefore conclude as above that the local ND map determines the elastic tensor and density at the boundary.
\end{proof}

\section{Interior determination}\label{inner}

For the purposes of this section, $\dom$ will be a domain in $\R^3$, $D\subset\dom$ a subdomain, $\sig\subset \pr D\cap\pr\dom\neq\emptyset$. Let $\dom_2=\dom\setminus \overline D$ and $\sig_2=\pr\dom_2\setminus\pr\dom$. Suppose $\cc$ and $\rho$ are homogeneous in $D$ and let $\Lambda^T_{\sig_2}$ be the DN map for the domain $\dom_2$ with data on $\sig_2$ and $\Phi^T_{\sig_2}$ be the similarly defined ND map. We will prove the following proposition.
\begin{prop}\label{interior-prop}
There exists some $0<\delta<\infty$ depending on $D$, $\cc|_{D}$, $\rho|_{D}$ such that
\begin{itemize}
    \item[(i)] $\Lambda^{T,\sig}_{\rho,\cc}$ determines $\Lambda^{T-2\delta}_{\sig_2}$, \item[(ii)] $\Phi^{T,\sig}_{\rho,\cc}$ determines $\Phi^{T-2\delta}_{\sig_2}$.
\end{itemize}
\end{prop}

We need the following  result from \cite{ET}: 
\begin{prop}\label{ucp}
Suppose $\cc$, $\rho$ are homogeneous in $D$. There is a (non-Riemannian) metric $N$ on $\mathcal{T} D$, determined by $\cc|_D$ and $\rho|_D$, such that if $P_{\rho,\cc}w=0$ in $D\times(0,T)$, $(w, \pr_\cc w)|_{(0,T)\times\sig}=0$, then $w(T/2,x)=0$ for any $x\in D$ such that $d_N(x,\sig)<T/2$.
\end{prop}
Here $N$ is a family of norms $N_x$ on $\R^3\equiv\mathcal{T}_x D$, $x\in\dom$, which induces a distance on $D$ by
\begin{equation}
d(x,y)=\inf_{\gamma}\int_0^1 N_{\gamma(t)}(\gamma'(t))\dd t,
\end{equation}
where the infimum is taken over curves $\gamma\in C^1([0,1];D)$ such that $\gamma(0)=x$ and $\gamma(1)=y$. The distance to the boundary is defined in the usual way as
\begin{equation}
d_N(x,\sig)=\inf_{y\in\sig}d_N(x,y).
\end{equation}

Now let $H^{-1}(\Omega)$ be the dual space of $H^1_0(\Omega)$, and for $F\in H^1((0,T); H^{-1}(\Omega))$ let $u$ satisfy
\begin{equation}\label{inner-data}
\left\{\begin{array}{l}
P_{\rho,\cc}u=F\,\,\text{\rm in}\,\,\dom_T,\\[5pt]
u|_{(0,T)\times\gam}=0,\\[5pt]
(u, \pr_t u)|_{t=0}=0. 
\end{array}\right.
\end{equation}
Then $u$ has the estimate 
\begin{equation}
||u||_{X_D}\leq C||F||_{H^1((0,T);H^{-1}(\dom))},
\end{equation}
with
\begin{equation}
 ||u||_{X_D}=||u||_{L^\infty((0,T);H^1_0(\dom))}+||\pr_t u||_{L^\infty((0,T); L^2(\dom))}.
\end{equation}
We write 
\begin{equation}
G^D_{\rho,\cc}(F)=u, \quad G^D_{\rho,\cc}:H^1((0,T);H^{-1}(\dom))\to X_D.
\end{equation}

Also, if $u$ satisfies
\begin{equation}\label{inner-data-N}
\left\{\begin{array}{l}
P_{\rho,\cc}u=F\,\,\text{\rm in}\,\,\dom_T,\\[5pt]
\pr_\cc u|_{(0,T)\times\gam}=0,\\[5pt]
(u, \pr_t u)|_{t=0}=0, 
\end{array}\right.
\end{equation}
then 
\begin{equation}
||u||_{X_N}\leq C||F||_{H^1((0,T);H^{-1}(\dom))},
\end{equation}
with
\begin{equation}
 ||u||_{X_N}=||u||_{L^\infty((0,T);H^1(\dom))}+||\pr_t u||_{L^\infty((0,T); L^2(\dom))}.
\end{equation}
We write 
\begin{equation}
G^N_{\rho,\cc}(F)=u, \quad G^N_{\rho,\cc}:H^1((0,T);H^{-1}(\dom))\to X_N.
\end{equation}

\bigskip
Now let $\delta=\sup_{x\in D} d_N(x,\sig)$. Then we have the following lemma.
\begin{lem}\label{green-determination}
For any $F\in C^\infty(\dom_T)$, $\supp F\subset (0, T-\delta)\times D$, the DN map $\Lambda^{T,\sig}_{\rho,\cc}$, $\rho|_D$, and $\cc|_D$ determine $G^D_{\rho,\cc}(F)|_{ (0, T-2\delta)\times D}$.
\end{lem}
\begin{proof}
For $f\in C_0^\infty((0,T)\times\sig)$ we define $S^T_{\sig}(f)=u|_{(0,T)\times D}$, where $u$ satisfies
\begin{equation}\label{boundary-data}
\left\{\begin{array}{l}
P_{\rho,\cc}u=0\,\,\text{\rm in}\,\,\dom_T,\\[5pt]
u|_{{(0,T)\times\gam}}=f,\\[5pt]
(u,\pr_t u)|_{t=0}=0.
\end{array}\right.
\end{equation}
 Similarly, if $F\in C_0^\infty((0,T)\times D)$ we define $S^T(F)=u|_{(0,T)\times D}$, where $u$ satisfies \eqref{inner-data}.\footnote{We use this two notations as defined here only within the proof of this lemma.} 

The first step in the proof is to show that $\Lambda^{T,\sig}_{\rho,\cc}$, $\rho|_D$, and $\cc|_D$ determine $S^{T-\delta}_\sig$. Let $f\in C_0^\infty((0,T)\times\sig)$ and  $u$ that satisfies \eqref{boundary-data}. Extend $u$ by $0$ to $(-\infty,0)\times\dom$. Let
\begin{multline}
\mathcal{S}_{f_1,f_2}^T=\{ v\in C^\infty((-\infty,T)\times D): P_{\rho,\cc} v=0, \\
v|_{(-\infty,T)\times \sig}=f_1, \pr_\cc v|_{(-\infty,T)\times \sig}=f_2\},
\end{multline}
which is determined by $f_1$, $f_2$, $D$, $\sig$, $\rho|_D$, and $\cc|_D$. Let $u'\in\mathcal{S}_{f,\Lambda^{T,\sig}_{\rho,\cc} f}^T$, and set $w=u-u'$. Then $P_{\rho,\cc}w=0$ on $D$,  and $(w, \pr_\cc w)|_{(0,T)\times\sig}=0$. For $t<T-\delta$ we can apply Proposition \ref{ucp} with the time interval $(t-\delta, t+\delta)$ to conclude that $w(t,x)=0$ for all $x\in D$. So we can now assume $S^{T-\delta}_\sig$ is known.

 Let $\mathfrak{r}$ be the time reversal operator on $(0,T-\delta)$. That is $(\mathfrak{r}\,\ell)(t)=\ell(T-\delta-t),\,t\in(0,T-\delta)$ for any function $\ell$ over $(0,T-\delta)$. If we write
\begin{equation}
u^*=\mathfrak{r}S^{T-\delta}_\sig\mathfrak{r}f,
\end{equation}
$u^*$ satisfies
\begin{equation}
\left\{\begin{array}{l}
P_{\rho,\cc}u^*=0\,\,\text{\rm in}\,\,(0, T-\delta)\times D,\\[5pt]
 u^*|_{(0,T)\times\gam}= f, (u^*, \pr_t u^*)|_{t=T-\delta}=0.
\end{array}\right.
\end{equation}
We want to identify the adjoint in $L^2((0,T-\delta)\times D)$ of $\mathfrak{r}S^{T-\delta}_\sig\mathfrak{r}$. Let $F\in C_0^\infty((0,T)\times D)$ and denote $v=G^D_{\rho,\cc}(F)$. Then, using integration by parts,
\begin{multline}
\la F, \mathfrak{r}S^{T-\delta}_\sig\mathfrak{r} f \ra=
\int_0^{T-\delta}\int_D F(t,x) u^*(t,x) \dd x\dd t=
\int_0^{T-\delta}\int_\dom P_{\rho,\cc}(v) u^* \dd x\dd t\\=
-\int_0^{T-\delta}\int_\sig \pr_\cc v f.
\end{multline}
We may then take the map $F\to \pr_\cc v|_\sig$ to be known. Extend $v$ as $0$ to $(-\infty,0)\times\dom$ and let $v'\in \mathcal{S}^{T-\delta}_{0, \pr_\cc v|_\sig}$. Define $w=v-v'$. As above we can conclude that $w(t,x)=0$ for all $t<T-2\delta$ and $x\in D$. This proves that $S^{T-2\delta}(F)$ is determined by the  knowledge of $D$, the DN map $\Lambda^{T,\sig}_{\rho,\cc}$, $\rho|_D$ and $\cc|_D$.
\end{proof}

\begin{lem}\label{green-determination-N}
For any $F\in C^\infty(\dom_T)$, $\supp F\subset (0, T-\delta)\times D$, the ND map $\Phi^{T,\sig}_{\rho,\cc}$, $\rho|_D$, and $\cc|_D$ determine $G^N_{\rho,\cc}(F)|_{ (0, T-2\delta)\times D}$.
\end{lem}
\begin{proof}
For $g\in C_0^\infty((0,T)\times\sig)$ we define $S^T_{\sig}(g)=u|_{(0,T)\times D}$, where $u$ satisfies
\begin{equation}\label{boundary-data-N}
\left\{\begin{array}{l}
P_{\rho,\cc}u=0\,\,\text{\rm in}\,\,\dom_T,\\[5pt]
\pr_\cc u|_{{(0,T)\times\gam}}=g,\\[5pt]
(u,\pr_t u)|_{t=0}=0.
\end{array}\right.
\end{equation}
 Similarly, if $F\in C_0^\infty((0,T)\times D)$ we define $S^T(F)=u|_{(0,T)\times D}$, where $u$ satisfies \eqref{inner-data-N}.

The first step in the proof is to show that $\Phi^{T,\sig}_{\rho,\cc}$, $\rho|_D$, and $\cc|_D$ determine $S^{T-\delta}_\sig$. Let $g\in C_0^\infty((0,T)\times\sig)$ and  $u$ that satisfies \eqref{boundary-data-N}. Extend $u$ by $0$ to $(-\infty,0)\times\dom$. Let $u'\in\mathcal{S}_{\Phi^{T,\sig}_{\rho,\cc} g, g}^T$, and set $w=u-u'$. Then $P_{\rho,\cc}w=0$ on $D$,  and $(w, \pr_\cc w)|_{(0,T)\times\sig}=0$. For $t<T-\delta$ we can apply Proposition \ref{ucp} with the time interval $(t-\delta, t+\delta)$ to conclude that $w(t,x)=0$ for all $x\in D$. So we may now assume $S^{T-\delta}_\sig$ is known.

If we write
\begin{equation}
u^*=\mathfrak{r}S^{T-\delta}_\sig\mathfrak{r}g,
\end{equation}
$u^*$ satisfies
\begin{equation}
\left\{\begin{array}{l}
P_{\rho,\cc}u^*=0\,\,\text{\rm in}\,\,(0, T-\delta)\times D,\\[5pt]
\pr_\cc  u^*|_{(0,T)\times\gam}= g,\\[5pt]
 (u^*, \pr_t u^*)|_{t=T-\delta}=0.
\end{array}\right.
\end{equation}
We want to identify the adjoint in $L^2((0,T-\delta)\times D)$ of $\mathfrak{r}S^{T-\delta}_\sig\mathfrak{r}$. Let $F\in C_0^\infty((0,T)\times D)$ and denote $v=G^N_{\rho,\cc}(F)$. Then, using integration by parts,
\begin{multline}
\la F, \mathfrak{r}S^{T-\delta}_\sig\mathfrak{r} f \ra=
\int_0^{T-\delta}\int_D F(t,x) u^*(t,x) \dd x\dd t=
\int_0^{T-\delta}\int_\dom P_{\rho,\cc}(v) u^* \dd x\dd t\\=
\int_0^{T-\delta}\int_\sig v\cdot g.
\end{multline}
We can then take the map $F\to  v|_\sig$ to be known. Extend $v$ as $0$ to $(-\infty,0)\times\dom$ and let $v'\in \mathcal{S}^{T-\delta}_{v|_\sig,0}$. Define $w=v-v'$. As above we can conclude that $w(t,x)=0$ for all $t<T-2\delta$ and $x\in D$. This proves that $S^{T-2\delta}(F)$ is determined by the knowledge of $D$, the ND map $\Phi^{T,\sig}_{\rho,\cc}$, $\rho|_D$ and $\cc|_D$.
\end{proof}

Let  $f\in C^1((0,T); H^{-1/2}(\sig_2))$, and define $T_f\in C^1((0,T); H^{-1}(\dom))$ by
\begin{equation}
\la T_f,\phi\ra=\int_0^T\la f,\phi|_{\sig_2}\ra,\quad \phi\in H^1_0(\dom).
\end{equation}
Define the operator  $\mathscr{L}^T_D$ by  $\mathscr{L}^T_D(f)=u|_{\sig_2}$, where
\begin{equation}\label{definition-LD}
\left\{\begin{array}{l}
P_{\rho,\cc}u=T_f\,\,\text{\rm in}\,\,\dom_T,\\[5pt]
 u|_{(0,T)\times\gam}=0,\\[5pt]
(u, \pr_t u)|_{t=0}=0,
\end{array}\right.
\end{equation}
and the operator $\mathscr{L}^T_N$ by $\mathscr{L}^T_N(f)=u|_{\sig_2}$, where
\begin{equation}\label{definition-LN}
\left\{\begin{array}{l}
P_{\rho,\cc}u=T_f\,\,\text{\rm in}\,\,\dom_T,\\[5pt]
 \pr_\cc u|_{(0,T)\times\gam}=0,\\[5pt]
(u, \pr_t u)|_{t=0}=0.
\end{array}\right.
\end{equation}

\begin{lem}\label{lemma33}
$\mathscr{L}^{T-2\delta}_D$ is determined by the  knowledge of $D$, the DN map $\Lambda^{T,\sig}_{\rho,\cc}$, $\rho|_D$, and $\cc|_D$.
$\mathscr{L}^{T-2\delta}_N$ is determined by the  knowledge of $D$, the ND map $\Phi^{T,\sig}_{\rho,\cc}$, $\rho|_D$, and $\cc|_D$.
\end{lem}
\begin{proof}
Let  $f\in C^2_0((0,T);C_0^\infty(\sig_2))$. In local (in space) coordinates we can arrange that $\sig_2$ is $\{x_3=0\}$ and $D$ is  $\{x_3>0\}$. Suppose the spatial support of $f$ is entirely in this coordinate patch. For $\epsilon>0$ define $F_\epsilon\in C^1((0,T);H^{-1}(\dom))$ by
\begin{equation}
 F_\epsilon (\phi)(t)= \int_{x_3=0} f(t,x')\phi(x',\epsilon)\dd x', \quad \phi\in C_0^\infty(D).
\end{equation}
Then
\begin{multline}
|F_\epsilon(\phi)(t)-T_f(\phi)(t)|= \left|\int_{0\leq x_3\leq\epsilon} f(t,x')\pr_3\phi(x',x_3)  \right|\\
\leq C||\phi||_{H^1_0(\dom)}||f||_{L^\infty((0,T)\times\sig_2)}\epsilon^{1/2},
\end{multline}
and similarly
\begin{equation}
|\pr_t F_\epsilon(\phi)(t)-\pr_tT_f(\phi)(t)| \leq C||\phi||_{H^1_0(\dom)}||\pr_t f||_{L^\infty((0,T)\times\sig_2)}\epsilon^{1/2}.
\end{equation}
Using a partition of unity argument, we can therefore construct a sequence $F_n\in C^1((0,T);H^{-1}(\dom))$, $\supp F_n\subset (0,T)\times D$, such that $F_n\to T_f$ in $C^1((0,T);H^{-1}(\dom))$. 

Suppose now that we take two functions $f,h \in C^2_0((0,T-2\delta);C_0(\sig_2))$ and construct sequences $F_n$, $H_n$ as above. By Lemma \ref{green-determination}, $H_n(G^D_{\rho,C}(F_n))$ is determined by the  knowledge of $D$, the DN map $\Lambda^{T,\sig}_{\rho, \cc}$, $\rho|_D$, and $\cc|_D$. Passing to the limit we see that so is $T_h( G^D_{\rho,\cc}(T_f))$ and therefore so is $\mathscr{L}^{T-2\delta}_D$. The same is true for $\mathscr{L}^{T-2\delta}_N$.
\end{proof}

 Let  $\Lambda^{T,+}_{\sig_2}$ be the DN map for the domain $D$ with data on $\sig_2$. Also let $\Phi^{T,+}_{\sig_2}$ be the ND map for the domain $D$ with data on $\sig_2$.

\begin{lem}
If $(\Lambda^T_{\sig_2}-\Lambda^{T,+}_{\sig_2})f=0$   for $f\in C^\infty_0((0,T);\hb(\sig_2))$, then $f=0$. If $(\Phi^T_{\sig_2}-\Phi^{T,+}_{\sig_2})g=0$ for $g\in C^\infty_0((0,T);\hbi(\sig_2))$, then $g=0$.
\end{lem}
\begin{proof}
Suppose $(\Lambda^T_{\sig_2}-\Lambda^{T,+}_{\sig_2})f=0$. Let $u$ and $u^+$ be the solutions in $\dom_2$ and $D$ respectively with Dirichlet data $f$ on $\sig_2$ and zero on the rest of their respective boundaries. Define $\tilde u= u\chi_{\dom_2}+u_+\chi_D$. By the assumption, since both Dirichlet and Neumann data accross $\sig_2$ match, we have that $P_{\rho,\cc} \tilde u=0$ in $\dom$. Since $\tilde u$ has zero initial Cauchy and lateral Dirichlet data, it must be zero.

In the case when $(\Phi^T_{\sig_2}-\Phi^{T,+}_{\sig_2})f=0$, the argument is identical.
\end{proof}

\begin{lem}\label{lem-surjective-D}
For $f\in C^\infty_0((0,T)\times\sig_2)$, 
\begin{equation}
(\Lambda^T_{\sig_2}-\Lambda^{T,+}_{\sig_2})\mathscr{L}^T_D f=f.
\end{equation}
\end{lem}
\begin{proof}
Let $u$ be as in \eqref{definition-LD} and $\phi\in C_0^\infty((0,T)\times\dom)$. Then
\begin{equation}
\la f,\phi|_{\sig_2}\ra=\int_0^T\int_\dom P_{\rho,\cc}(u) \phi=
\int_0^T\int_\dom \big(\pr_t^2 u \phi+Du:(\cc:D\phi)\big)
\end{equation}
and
\begin{multline}
\int_0^T\int_\dom Du:(\cc:D\phi) =\int_0^T\int_{\dom_2} Du:(\cc:D\phi)+\int_0^T\int_D Du:(\cc:D\phi)\\=
-\int_0^T\int_{\dom_2}L_{\cc}(u)\phi-\int_0^T\int_{D}L_{\cc}(u)\phi\\
+\la \Lambda^T_{\sig_2}\mathscr{L}^T_D f,\phi|_{\sig_2}\ra-\la \Lambda^{T,+}_{\sig_2}\mathscr{L}^T_D f,\phi|_{\sig_2}\ra.
\end{multline}
\end{proof}

\begin{proof}[Proof of Proposition \ref{interior-prop} (i)]It follows that $\Lambda^{T-2\delta}_{\sig_2}-\Lambda^{T-2\delta,+}_{\sig_2}$ is the inverse of $\mathscr{L}^{T-2\delta}_D$, which is then determined by $\Lambda^{T,\sig}_{\rho,\cc}$. The claim (i) follows as $\Lambda_{\Sigma_2}^{T-2\delta,+}$ is clearly determined by $D$, $C_D$
and $\rho|_D$.
\end{proof}

The argument for the ND map case is a little bit more involved. Before stating the lemma for ND maps that is analogous to Lemma \ref{lem-surjective-D} we introduce two more notations. If $u=G^N_{\rho,\cc}(T_f)$, then $\pr_{\cc_D}u|_{\sig_2^+}$ will denote the restriction of $\pr_\cc u$ to $\sig_2$, taken from the $D$ side of $\sig_2$. Similarly, $\pr_{\cc_{\dom_2}}u|_{\sig_2}$ will be the restriction of the same function to $\sig_2$, taken from the $\dom_2$ side of $\sig_2$. Note that here the unit normal vector in all cases points away from $\dom_2$ and into $D$.

\begin{lem}\label{lem-surjective-N}
For $f\in C^\infty_0((0,T)\times\sig_2)$,
\begin{equation}
(\Phi^T_{\sig_2}-\Phi^{T,+}_{\sig_2})(\pr_{\cc|_D}u|_{\sig_2^+}+f)=-\Phi^{T,+}_{\sig_2}(f).
\end{equation}
\end{lem}
\begin{proof}
As in the proof of Lemma \ref{lem-surjective-D}, for $\phi\in C_0^\infty((0,T)\times\dom)$
\begin{equation}
\la f,\phi|_{\sig_2}\ra=\int_0^T\int_\dom P_{\rho,\cc}(u) \phi=
\int_0^T\int_\dom\big( \pr_t^2 u \phi+Du:(\cc:D\phi)\big),
\end{equation}
and
\begin{multline}
\int_0^T\int_\dom Du:(\cc:D\phi) =\int_0^T\int_{\dom_2} Du:(\cc|_{\dom_2}:D\phi)+\int_0^T\int_D Du:(\cc|_D:D\phi)\\=
-\int_0^T\int_{\dom_2}L_{\cc}(u)\phi-\int_0^T\int_{D}L_{\cc}(u)\phi\\
+\la \pr_{\cc_{\dom_2}}u|_{\sig_2},\phi|_{\sig_2}\ra-\la \pr_{\cc_D}u|_{\sig_2^+},\phi|_{\sig_2}\ra.
\end{multline}
It follows that
\begin{equation}
\pr_{\cc_{\dom_2}}u|_{\sig_2}=\pr_{\cc_D}u|_{\sig_2^+}+f.
\end{equation}
Since $u(t,\cdot)\in H^1(\dom)$, its Dirichlet data on each side of $\sig_2$ coincide. Therefore
\begin{equation}
\Phi^T_{\sig_2}(\pr_{\cc_{\dom_2}}u|_{\sig_2})=\Phi^{T,+}_{\sig_2}(\pr_{\cc_{D}}u|_{\sig_2^+}).
\end{equation}
The conclusion follows immediately.
\end{proof}

The surjectivity of the difference $\Phi^T_{\sig_2}-\Phi^{T,+}_{\sig_2}$ follows from the next lemma.

\begin{lem}\label{phi-surjective}
The local ND map 
\begin{equation}
\Phi^{T,+}_{\sig_2}: C_0^\infty((0,T];\hbi(\sig_2))\to C_0^\infty((0,T];H^{1/2}(\sig_2))
\end{equation}
 is surjective.
\end{lem}

First note that since the coefficients of the equation are time independent, taking time derivatives of all orders and applying Theorem \ref{thm-forward}, we have the following corollary with the same notations as in Theorem \ref{thm-forward}.
\begin{cor}
If $f\in C_0^\infty((0,T];H^{1/2}(\Gamma_u))$, $g\in C_0^\infty((0,T];H^{-1/2}(\Gamma_s))$, 
then there exists a unique $u\in C_0^\infty((0,T];H^1(\dom))$ such that
\begin{equation}
\left\{\begin{array}{l}
P_{\rho,\cc}u=0\,\,\text{\rm in}\,\,\dom,\\[5pt]
u|_{(0,T)\times\Gamma_u}=f,\\[5pt]
\pr_\cc u|_{(0,T)\times\Gamma_s}=g,\\[5pt]
(u,\pr_t u)|_{t=0}=0.
\end{array}\right.
\end{equation}
\end{cor}
In particular, taking $\Gamma_u = \emptyset$ and $\Gamma_s = \Gamma$, this justifies the spaces in between which $\Phi^{T,+}_{\sig_2}$ maps in the statement of the lemma.

\begin{proof}[Proof of Lemma \ref{phi-surjective}]
To prove the surjectivity, let $f\in C_0^\infty((0,T];H^{1/2}(\sig_2))$. We choose $\gam_u=\sig_2$, $\gam_s=\pr D\setminus\overline{\sig_2}$. Let $u\in C_0^\infty((0,T];H^1(\dom))$ be such that
\begin{equation}
\left\{\begin{array}{l}
P_{\rho,\cc}u=0\,\,\text{\rm in}\,\,D,\\[5pt]
u|_{(0,T)\times\Gamma_u}= f,\\[5pt]
\pr_\cc u|_{(0,T)\times\Gamma_s}=0,\\[5pt]
(u,\pr_t u)|_{t=0}=0.
\end{array}\right.
\end{equation}
Since $\Phi^{T,+}_{\sig_2}(\pr_{\cc|_D} u|_{\sig_2^+})=f$, we have our desired conclusion.
\end{proof}

\begin{proof}[Proof of Proposition \ref{interior-prop} (ii)]With the same notation used in Lemma \ref{lem-surjective-N}, for $f\in C^\infty_0((0,T);H^{-1/2}(\sig_2))$ define the operator
\begin{equation}
\begin{array}{c}
\mathscr{K}:C^\infty_0((0,T);\hbi(\sig_2))\to C^\infty_0((0,T);\hbi(\sig_2)),\\[5pt]
\mathscr{K}(f)=\pr_{\cc|_D}u|_{\sig_2^+}+f.
\end{array}
\end{equation}
By Lemmas \ref{lem-surjective-N} and \ref{phi-surjective}, this operator is surjective. It follows from the proof of Lemma \ref{lemma33} that $\mathscr{K}$ is  determined by the knowledge of $\dom$, $D$, $\cc|_D$, $\rho|_D$, and $\Phi^{T,\sig}_{\rho, \cc}$. Its right inverse, which we will denote by $\kappa$, is then determined by the same quantities. 
We have now that $\Phi^T_{\sig_2}-\Phi^{T,+}_{\sig_2}=-\Phi^{T,+}_{\sig_2}\circ\kappa$ is determined by the knowledge of $\dom$, $D$, $\cc|_D$, $\rho|_D$, and $\Phi^{T,\sig}_{\rho, \cc}$. This ends the proof.
\end{proof}

\section{Proofs of Theorem \ref{main-thm} and Theorem \ref{thm-interest}}\label{proof}

\begin{proof}[Proof of Theorem \ref{main-thm}]The time $T_1$ can be chosen arbitrarily small. By Proposition \ref{boundary-prop}, if $\Lambda^{T_1,\sig}_{\rho^{(1)},C^{(1)}}=\Lambda^{T_1,\sig}_{\rho^{(2)},C^{(2)}}$ or $\Phi^{T_1,\sig}_{\rho^{(1)},C^{(1)}}=\Phi^{T_1,\sig}_{\rho^{(2)},C^{(2)}}$, then $\rho^{(1)}|_{D_1}=\rho^{(2)}|_{D_1}$ and $C^{(1)}|_{D_1}=C^{(2)}|_{D_1}$. Now let $\delta$ be as in Proposition \ref{interior-prop}, with $D=D_1$. We can choose $T_2=T_1+2\delta$. If $\Lambda^{T_2,\sig}_{\rho^{(1)},C^{(1)}}
=\Lambda^{T_2,\sig}_{\rho^{(2)},C^{(2)}}$ or $\Phi^{T_2,\sig}_{\rho^{(1)},C^{(1)}}
=\Phi^{T_2,\sig}_{\rho^{(2)},C^{(2)}}$, then $\Lambda^{T_1,\gam_2}_{\rho^{(1)},C^{(1)}}=\Lambda^{T_1,\gam_2}_{\rho^{(2)},C^{(2)}}$ or $\Phi^{T_1,\gam_2}_{\rho^{(1)},C^{(1)}}=\Phi^{T_1,\gam_2}_{\rho^{(2)},C^{(2)}}$, respectively, where these DN and ND maps are taken relative to $\dom\setminus \overline{D_1}$. By Proposition \ref{boundary-prop} it follows that $\rho^{(1)}|_{D_2}=\rho^{(2)}|_{D_2}$ and $C^{(1)}|_{D_2}=C^{(2)}|_{D_2}$. It is clear that we may continue in this way to inductively construct all the times $T_k$.
\end{proof}

\begin{proof}[Proof of Theorem \ref{thm-interest}]
It is clear that we only need to prove the result in the case $R=\dom$. Notice that $\rho^{(I)}$, $C^{(I)}$ are all constant on all the elements of the common partition 
\begin{equation}
\{\tilde D_\gamma\}=\{D_\alpha^{(1)}\cap D_\beta^{(2)}\}.
\end{equation}
 Let $P\in R$ and pick a smooth curve $\omega:[0,1]\to R$ so that $\omega(0)\in\sig$, $\omega(1)=P$, and which intersects the boundaries of the subdomains $\tilde D_\gamma$ only at smooth points and transversally. Let $V_\epsilon$ be the tubular neighborhood of $\omega([0,1])$ of radius $\epsilon>0$. We can choose $\epsilon$ small enough that $V_\epsilon$ only intersects the smooth components of the boundaries of the subdomanins $\tilde D_\gamma$ and $\pr V_\epsilon$ is transversal to all of them. It then follows that each $D'_\gamma=\tilde D_\gamma\cap V_\epsilon$ is a Lipschitz set. 
 
 We can label by $D'_1$, \ldots, $D'_N$ the chain of non-empty sets in $\{D'_\gamma\}$, in the order in which the curve $\omega$ intersects them. This chain satisfies the conditions of Theorem \ref{main-thm} and we may conclude that there is a time $0<T_P<\infty$ such that if $\Lambda^{T_P,\sig}_{\rho^{(1)},C^{(1)}}= \Lambda^{T_P,\sig}_{\rho^{(2)},C^{(2)}}$ or $\Phi^{T_P,\sig}_{\rho^{(1)},C^{(1)}}= \Phi^{T_P,\sig}_{\rho^{(2)},C^{(2)}}$, then $\rho^{(1)}(P)=\rho^{(2)}(P)$ and $C^{(1)}(P)=C^{(2)}(P)$. Choosing $T=\max_{P\in R}T_P<\infty$, we have the result.
\end{proof}

\noindent
\bigskip{\bf\large Acknowledgement}${}$
\newline
This project began while the first author was employed at the Hong Kong University of Science and Technology, Jockey Club Institute for Advanced Study, and concluded while he was visiting Hokkaido University, partially supported by Sichuan University.
The second author was partially supported by Grant-in-Aid for Scientific Research (15K21766, 15H05740) of the Japan Society for  the  Promotion  of  Science  doing  the  research  of  this  paper. The third author was supported by EPSRC grants EP/P01593X/1 and EP/R002207/1.

\appendix

\section{Sub-analytic sets}\label{sas}

In this appendix, for the convenience of the reader, we give the definition and summarize a few of the properties of sub-analytic sets.

Let $X$ be a real analytic manifold.
A set $A\subset X$ is semi-analytic if for any $x\in \overline{A}$ (here $\overline{A}$
denotes the closure of $A$) there exists an open neighborhood $U$ of $x$ in $X$ and finitely many real-analytic functions $f_{ij}:U\to\R$, $i=1,\ldots,p$, $j=1,\ldots,q$,  such that
\begin{equation}
A\cap U=\bigcup_{i=1}^p\bigcap_{j=1}^q \{x\in U:f_{ij}(x) \ast_{ij} 0\},
\end{equation}
where the relations $\ast_{ij}$ are either ``$>$'' or ``$=$''.
  For example, a finite union of linear or curved polyhedra in $\R^n$, whose boundaries are level sets of real-analytic functions, is a semi-analytic set. A good reference for semi-analytic sets is \cite{Bierstone-Milman}.

Now we introduce the notion of a subanalytic set, which is just obtained
in the above definition
by replacing subsets determined by inequalities 
with the ones of images of analytic maps.
That is, $A$ is said to be subanalytic if for any $x \in \overline{A}$ there exist
an open neighborhood $U$ of $x$, real analytic compact manifolds $Y_{i,j}$,
 $i=1,2,\,1\le j\le N$ and real
analytic maps $\Phi_{i,j}:Y_{i,j}\rightarrow X$ such that
\begin{equation}
A\cap U=\bigcup_{j=1}^N(\Phi_{1,j}(Y_{1,j})\setminus\Phi_{2,j}(Y_{2,j}))\,\, \bigcap\,\, U.
\end{equation}
Reference is made to \cite{Bierstone-Milman} and \cite{Kashiwara-Schapira},
where we can find all the required proofs for properties stated below: A family
of subanalytic sets is stable under several set theoretical operations.
Note that, by definition, a semi-analytic subset is subanalytic.
\begin{enumerate}
\item A finite union and a finite intersection of subanalytic subsets are subanalytic.
\item The closure, interior and complement of a subanalytic subset
are again subanalytic. In particular, its boundary is subanalytic.
\item The inverse image of a subanalytic set by an analytic map is subanalytic.
Further, the direct image of a subanalytic set by a proper analytic map
is also subanalytic.
\end{enumerate}

The other important properties needed in this paper 
are the following  ``finiteness property'' and ``triangulation theorem'' of a subanalytic set.
\begin{lem}[Theorem 3.14 \cite{Bierstone-Milman}]{\label{lem:finiteness}}
	Each connected component of a subanalytic set is subanalytic.
Furthermore, connected components of a subanalytic set are locally finite, that is,
for any compact subset $K$ and a subanalytic subset $A$,
the number of connected components of $A$ intersecting $K$ is finite.
\end{lem}

In particular, for two relatively compact subanalytic subsets $A$ and $B$,
the number of connected components of $A \cap B$ is always finite.

\begin{lem}[Proposition 8.2.5 \cite{Kashiwara-Schapira}]{\label{lemma:triangulation}}
	Let $X= \underset{\lambda \in \Lambda}{\sqcup} X_\lambda$ be a locally finite partition
	of $X$ by subanalytic subsets. Then there exist a simplicial complex
	$\mathbf{S} = (S,\Delta)$ and a homeomorphism $i: |\mathbf{S}| \to X$ such that
\begin{enumerate}
	\item for any simplex $\sigma \in \Delta$, the image $\hat{\sigma} := i(|\sigma|)$ is subanalytic in $X$
		and real analytic smooth at every point in $\hat{\sigma}$.
\item for any simplex $\sigma \in \Delta$, there exists $\lambda \in \Lambda$
	with $i(|\sigma|) \subset X_\lambda$.
\end{enumerate}
\end{lem}

\bibliography{hyperbolic}
\bibliographystyle{plain}
\end{document}